\documentclass[12pt,a4paper,twoside]{article}
\usepackage[colorlinks=true,linkcolor=blue,citecolor=blue]{hyperref}
\usepackage{amsmath,mathrsfs,multicol,amsfonts,mathtools,amsthm, bm,fancyhdr,lastpage,xcolor,titlesec,cleveref}
\usepackage[english]{babel}
\usepackage{nicematrix}
\usepackage{amssymb}
\usepackage[bottom,multiple]{footmisc}
\usepackage[top=2.5cm,bottom=2.5cm,left=2.5cm,right=2.5cm]{geometry}
\titleformat{\section}[block]{\bfseries\large}{\thesection. }{2pt}{}
\theoremstyle{definition}
\newtheorem{theorem}{Theorem}[section]

\newtheorem{proposition}{Proposition}[section]
\newtheorem{corollary}[theorem]{Corollary}

\newtheorem{remark}[theorem]{Remark}

\begin{document}
\thispagestyle{empty}
\begin{center}
\noindent\textbf{{\Large Isometry groups of simply connected unimodular 4-dimensional Lie groups}}
\end{center}
\begin{center} \noindent Youssef Ayad\footnote{corresponding author: youssef.ayad@edu.umi.ac.ma} and Said Fahlaoui\footnote{s.fahlaoui@umi.ac.ma}\\
	\small{\textit{$^{1, 2}$Faculty of sciences, Moulay Ismail University of Meknes}}\\
	\small{\textit{B.P. 11201, Zitoune, Meknes}, Morocco}
\end{center}
\begin{center}
\textbf{Abstract}
\end{center}
\begin{center}
\hspace*{3mm}    We describe the full group of isometries of each left invariant Riemannian metric on the simply connected unimodular nilpotent or solvable $(R)$-type Lie groups of dimension four.
\end{center}
\begin{center}
\end{center}
\begin{center}
\end{center}
\section{Introduction}
\hspace*{3mm}     The main goal of this paper is to address the following problem 
\begin{center}
	Determine the full isometry group of each left invariant Riemannian metric on the simply connected unimodular nilpotent or solvable $(R)$-type Lie groups of dimension four.
\end{center}
There are exactly seven such $4$-dimensional Lie groups which are \cite{van2017metrics}
$$\operatorname{Nil}^3 \times \mathbb{R} \qquad \operatorname{Nil}^4 \qquad \operatorname{Sol_{m, n}^4} \qquad \operatorname{Sol}^3 \times \mathbb{R} \qquad \operatorname{Sol}_0^4 \qquad \operatorname{Sol}_0^{' 4} \qquad \operatorname{Sol}_1^4$$
where the two Lie groups $\operatorname{Nil}^3 \times \mathbb{R}$ and $\operatorname{Nil}^4$ are nilpotent and the others are solvable $(R)$-type Lie groups. To understand why the Lie groups $\operatorname{Sol_{m, n}^4}, \operatorname{Sol}^3 \times \mathbb{R}, \operatorname{Sol}_0^4, \operatorname{Sol}_0^{' 4}, \operatorname{Sol}_1^4$ are of type $(R)$, see the abstract and Theorems 4.1, 4.2, 4.3, 4.4, 5.1 of the paper \cite{van2017metrics}. The classification of left invariant Riemannian metrics on these Lie groups is due to Van Thuong \cite{van2017metrics}. There is a complete study of the two problems of classifying left-invariant Riemannian metrics, and determining their isometry groups, in the three-dimensional Lie groups given by \cite{ha2009left,shin1997isometry,ha2012isometry} in the unimodular case and by \cite{ana2022isometry} in the non-unimodular case. See also \cite{biggs25isometries} for other general remarks.\\
Let $G$ be a connected and simply connected Lie group with Lie algebra $\mathfrak{g}$ and let $g$ be a left invariant Riemannian metric on $G$. The group of isometries of $G$ with respect to the left invariant Riemannian metric $g$ denoted $\operatorname{Isom}(G, g)$, is the group of all diffeomorphisms $\theta$ of $G$ such that the pullback of $g$ by $\theta$ is equal to $g$. This latter is a Lie group under the compact-open topology and acts on $G$ transitively \cite{myers1939group}. The isotropy subgroup at the identity element $e$ of $G$ is denoted by $\operatorname{Isom}(G, g)_{e}$, it consists of isometries of $G$ fixing the identity element $e$. Therefore, one simply gets that
$$\operatorname{Isom}(G, g) = L(G) \cdot \operatorname{Isom}(G, g)_{e} \cong G \cdot \operatorname{Isom}(G, g)_{e}$$
where $L(G)$ is the subgroup of left translations which can be identified with $G$. In fact, if $\theta$ is an element of $\operatorname{Isom}(G, g)$ such that $\theta(e) = p \in G$, then $\theta$ decomposes as
$$\theta = L_{p} \circ (L_{p^{-1}} \circ \theta) \quad \text{where}\,\, L_{p^{-1}} \circ \theta \in \operatorname{Isom}(G, g)_{e}.$$
The product $L(G) \cdot \operatorname{Isom}(G, g)_{e}$ is not in general a semidirect product, it depends on when $L(G)$ is normal in $\operatorname{Isom}(G, g)$. In \cite{wolf1962locally,wilson1982isometry}, it was proved that if $G$ is nilpotent then
\begin{equation}
	\operatorname{Isom}(G, g) \cong G \rtimes \operatorname{Aut}(G)_{g}. \label{formula1}
\end{equation}
where $\operatorname{Aut}(G)_{g}$ is the group of isometric automorphisms of $G$ which is described below. The isometry groups of $4$-dimensional nilpotent Lie groups with respect to a family of Lorentzian metrics are studied in \cite{vsukilovic2017isometry}.\\
Since $G$ is simply connected, then $\operatorname{Aut}(G) \backsimeq \operatorname{Aut}(\mathfrak{g})$ and we have an action of this group on the set $\mathscr{L}$ of all left invariant Riemannian metrics on $G$ given by
$$ \begin{array}{rcl}
\operatorname{Aut}(\mathfrak{g}) \times \mathscr{L}&\longrightarrow& \mathscr{L}\\
(\theta, g) &\longmapsto& \theta^{\ast}g = g_{\theta}
\end{array} $$
where $g_{\theta}(u, v) = g(\theta^{-1}u, \theta^{-1}v) \,\, \forall u, v\in \mathfrak{g}$. See \cite{ha2012isometry} for more detail.\\
We denote the isotropy subgroup of $\operatorname{Aut}(\mathfrak{g})$ at $g$ by $\operatorname{Aut}(\mathfrak{g})_{g} = \left\lbrace \theta \in \operatorname{Aut}(\mathfrak{g}) / \theta^{\ast}g = g \right\rbrace $.\\
The group $G$ is called of type $(R)$, if for every $x \in \mathfrak{g}$, the endomorphism $\operatorname{ad} x : \mathfrak{g} \longrightarrow \mathfrak{g}$ has only real eigenvalues. As the result $(1)$ is important, The following is also a very important result: Let $G$ be a simply connected, unimodular, solvable Lie group of type (R), then for each left invariant Riemannian metric $g$ on $G$
\begin{equation}
	\operatorname{Isom}(G, g) \cong G \rtimes \operatorname{Aut}(G)_g. \label{formula2} \qquad (\text{See page 192 in}\;\cite{ha2012isometry})
\end{equation}
This result is a consequence of two results given by Gordon with Wilson and Helgason in \cite{gordon1988isometry,helgason}. In the case of the four-dimensional oscillator group, which is unimodular and solvable, but not of type $(R)$, the problem of determining the group of isometries is resolved in the following paper \cite{ayad2024}. To apply the previous two results, we need the automorphism group of our seven simply connected unimodular $4$-dimensional Lie groups. The automorphism group of four-dimensional Lie groups has been studied by several authors, including \cite{christodoulakis2003automorphisms, fisher2013automorphisms, biggs2016classification, van2017metrics}. From the description of the automorphism group of the simply connected unimodular Lie groups of dimension four, given in \cite{van2017metrics}, we provide the groups $\operatorname{Aut}(G)_g$ and then describe the full group of isometries of our four-dimensional Lie groups.
\section{Metrics on 4-dimensional unimodular Lie groups}
A left invariant Riemannian metric on a Lie group $G$ is simply an inner product $\langle ., .\rangle$ on the tangent space of $G$, which varies smoothly across $G$ and such that the left translations are isometries. This means that 
$$\langle u, v\rangle_{p} = \langle d_{p} L_{a}(u), d_{p} L_{a}(v)\rangle_{ap} \quad \forall p, a \in G \quad \forall u, v \in T_{p}G.$$
There is a bijective correspodence between left invariant Riemannian metrics on a Lie group $G$, and inner products on the Lie algebra $\mathfrak{g}$ of $G$. In fact, if $\langle ., .\rangle$ is an inner product on $\mathfrak{g}$, put 
$$\langle u, v\rangle_{p} = \langle d_{p} L_{p^{-1}}(u), d_{p} L_{p^{-1}}(v)\rangle \quad \forall p \in G \quad \forall u, v \in T_{p}G.$$
Then $\langle ., .\rangle_{p}$ defines a left invariant Riemannian metric on $G$. Thus, the classification of left invariant Riemannian metrics on a Lie group $G$, is equivalent to the classification of inner products on its Lie algebra $\mathfrak{g}$.
\begin{proposition} \cite{van2017metrics} 
	The set of inner products on $\mathfrak{g}$ is $(n^2 + n)/2$-dimensional, and can be identified with the set of upper triangular matrices with positive diagonal entries.
\end{proposition}
Following \cite{van2017metrics}, we give all left invariant Riemannian metrics on simply connected unimodular nilpotent or solvable $(R)$-type $4$-dimensional Lie groups. In fact, by \cite{van2017metrics}, the bijection between the set of upper triangular matrices with positive diagonal entries and the set of symmetric positive definite matrices is as follows:
$$ \begin{array}{rcl}
\varphi : \operatorname{Tsup}_{n}&\longrightarrow& \operatorname{S}_{n}\\
B &\longmapsto& (B^{-1})^T(B^{-1})
\end{array} $$
where \(\operatorname{Tsup}_{n}\) refers to the set of upper triangular matrices with positive diagonal entries, and \(\operatorname{S}_{n}\) refers to the set of symmetric positive definite matrices.
\subsection{Left invariant Riemannian metrics on $\operatorname{Nil}^3 \times \mathbb{R}$}
By theorem 3.1 in \cite{van2017metrics}, any metric on $\operatorname{Nil}^3 \times \mathbb{R}$ is equivalent, up to isometry, to the following metric
$$M = \begin{bmatrix}
\alpha & 0 & 0 & 0\\
0 & 1 & 0 & 0\\
0 & 0 & 1 & 0\\
0 & 0 & 0 & 1
\end{bmatrix} \quad \alpha > 0.$$
According to the bijection $\varphi$, the left invariant Riemannian metric (symmetric positive definite matrix) associated to $M$ is 
$$\langle ., .\rangle = \left( M^{-1}\right)^T\left( M^{-1}\right)  = \begin{bmatrix}
\frac{1}{\alpha^2} & 0 & 0 & 0\\
0 & 1 & 0 & 0\\
0 & 0 & 1 & 0\\
0 & 0 & 0 & 1
\end{bmatrix}.$$
\subsection{Left invariant Riemannian metrics on $\operatorname{Nil}^4$}
By theorem 3.2 in \cite{van2017metrics}, any metric on $\operatorname{Nil}^4$ is equivalent, up to isometry, to the following metric
$$M = \begin{bmatrix}
\alpha & \beta & 0 & 0\\
0 & \gamma & 0 & 0\\
0 & 0 & 1 & 0\\
0 & 0 & 0 & 1
\end{bmatrix} \quad \alpha, \gamma > 0 \quad \beta \geq 0.$$
We distinguish two cases associated to this metric
\begin{enumerate}
	\item If $\beta = 0$, then the left invariant Riemannian metric associated to $M$ is 
	$$\langle ., .\rangle_1 = \left( M^{-1}\right)^T\left( M^{-1}\right) = \begin{bmatrix}
	\frac{1}{\alpha^2} & 0 & 0 & 0\\
	0 & \frac{1}{\gamma^2} & 0 & 0\\
	0 & 0 & 1 & 0\\
	0 & 0 & 0 & 1
	\end{bmatrix}$$
	\item If $\beta > 0$, then our left invariant Riemannian metric is
	$$\langle ., .\rangle_2 = \left( M^{-1}\right)^T\left( M^{-1}\right) = \begin{bmatrix}
	\frac{1}{\alpha^2} & \frac{-\beta}{\alpha^2\gamma} & 0 & 0\\
	\frac{-\beta}{\alpha^2\gamma} & \frac{1}{\gamma^2} + \frac{\beta^2}{\alpha^2\gamma^2} & 0 & 0\\
	0 & 0 & 1 & 0\\
	0 & 0 & 0 & 1
	\end{bmatrix}.$$
\end{enumerate}
\subsection{Left invariant Riemannian metrics on $\operatorname{Sol_{m, n}^4}$}
By theorem 4.3 in \cite{van2017metrics}, any metric on $\operatorname{Sol_{m, n}^4}$ is equivalent, up to isometry, to the following metric
$$M = \begin{bmatrix}
1 & \alpha & \beta & 0\\
0 & 1 & \gamma & 0\\
0 & 0 & 1 & 0\\
0 & 0 & 0 & \mu
\end{bmatrix} \quad \mu > 0 \quad \alpha, \gamma \geq 0 \quad \beta \in \mathbb{R}.$$
We distinguish three cases associated to this metric (the other cases are similar to these cases in the sens that the isometry group is isomorphic to the one of theses metrics)
\begin{enumerate}
	\item If $\alpha = \beta = \gamma = 0$, then the left invariant Riemannian metric associated to $M$ is 
	$$\langle ., .\rangle_1 = \left( M^{-1}\right)^T\left( M^{-1}\right) = \begin{bmatrix}
	1 & 0 & 0 & 0\\
	0 & 1 & 0 & 0\\
	0 & 0 & 1 & 0\\
	0 & 0 & 0 & \frac{1}{\mu^2}
	\end{bmatrix}$$
	\item If $\alpha > 0$ and $\beta = \gamma = 0$, then our left invariant Riemannian metric is
	$$\langle ., .\rangle_2 = \left( M^{-1}\right)^T\left( M^{-1}\right) = \begin{bmatrix}
	1 & -\alpha & 0 & 0\\
	-\alpha & 1 + \alpha^2 & 0 & 0\\
	0 & 0 & 1 & 0\\
	0 & 0 & 0 & \frac{1}{\mu^2}
	\end{bmatrix}$$
	\item If $\alpha > 0$, $\beta \neq 0$ and $\gamma \in \mathbb{R}$, then our left invariant Riemannian metric is
	$$\langle ., .\rangle_3 = \begin{bmatrix}
	1 & -\alpha & \alpha\gamma - \beta & 0\\
	\\
	-\alpha & 1 + \alpha^2 & -\alpha(\alpha\gamma - \beta) - \gamma & 0\\
	\\
	\alpha\gamma - \beta & -\alpha(\alpha\gamma - \beta) - \gamma & (\alpha\gamma - \beta)^2 + \gamma^2 + 1 & 0\\
	\\
	0 & 0 & 0 & \frac{1}{\mu^2}
	\end{bmatrix}.$$
\end{enumerate}
\subsection{Left invariant Riemannian metrics on $\operatorname{Sol}^3 \times \mathbb{R}$}
By theorem 4.4 in \cite{van2017metrics}, any metric on $\operatorname{Sol}^3 \times \mathbb{R}$ is equivalent, up to isometry, to the following metric
$$M = \begin{bmatrix}
1 & \alpha & \beta & 0\\
0 & 1 & \gamma & 0\\
0 & 0 & 1 & 0\\
0 & 0 & 0 & \mu
\end{bmatrix} \quad \mu > 0 \quad \alpha, \gamma \geq 0 \quad \beta \in \mathbb{R}.$$
Hence this case is similar to the precedent one, we distinguish the following left invariant Riemannian metrics
$$\langle ., .\rangle_1 = \begin{bmatrix}
1 & 0 & 0 & 0\\
0 & 1 & 0 & 0\\
0 & 0 & 1 & 0\\
0 & 0 & 0 & \frac{1}{\mu^2}
\end{bmatrix} \qquad \langle ., .\rangle_2 = \begin{bmatrix}
1 & -\alpha & 0 & 0\\
-\alpha & 1 + \alpha^2 & 0 & 0\\
0 & 0 & 1 & 0\\
0 & 0 & 0 & \frac{1}{\mu^2}
\end{bmatrix}$$ $$\langle ., .\rangle_3 = \begin{bmatrix}
1 & -\alpha & \alpha\gamma - \beta & 0\\
\\
-\alpha & 1 + \alpha^2 & -\alpha(\alpha\gamma - \beta) - \gamma & 0\\
\\
\alpha\gamma - \beta & -\alpha(\alpha\gamma - \beta) - \gamma & (\alpha\gamma - \beta)^2 + \gamma^2 + 1 & 0\\
\\
0 & 0 & 0 & \frac{1}{\mu^2}
\end{bmatrix}.$$
\subsection{Left invariant Riemannian metrics on $\operatorname{Sol}_0^4$}
By theorem 4.1 in \cite{van2017metrics}, any metric on $\operatorname{Sol}_0^4$ is equivalent, up to isometry, to the following metric
$$M = \begin{bmatrix}
1 & 0 & \alpha & 0\\
0 & 1 & 0 & 0\\
0 & 0 & 1 & 0\\
0 & 0 & 0 & \beta
\end{bmatrix} \quad \alpha \geq 0 \quad \beta > 0.$$
We distinguish two cases associated to this metric
\begin{enumerate}
	\item If $\alpha = 0$, then the left invariant Riemannian metric associated to $M$ is 
	$$\langle ., .\rangle_1 = \left( M^{-1}\right)^T\left( M^{-1}\right) = \begin{bmatrix}
	1 & 0 & 0 & 0\\
	0 & 1 & 0 & 0\\
	0 & 0 & 1 & 0\\
	0 & 0 & 0 & \frac{1}{\beta^2}
	\end{bmatrix}$$
	\item If $\alpha > 0$, then our left invariant Riemannian metric is
	$$\langle ., .\rangle_2 = \left( M^{-1}\right)^T\left( M^{-1}\right) = \begin{bmatrix}
	1 & 0 & -\alpha & 0\\
	0 & 1 & 0 & 0\\
	-\alpha & 0 & 1 + \alpha^2 & 0\\
	0 & 0 & 0 & \frac{1}{\beta^2}
	\end{bmatrix}.$$
\end{enumerate}
\subsection{Left invariant Riemannian metrics on $\operatorname{Sol}_0^{' 4}$}
By theorem 4.2 in \cite{van2017metrics}, any metric on $\operatorname{Sol}_0^{' 4}$ is equivalent, up to isometry, to the following metric
$$M = \begin{bmatrix}
1 & 0 & \beta & 0\\
0 & \alpha & \gamma & 0\\
0 & 0 & 1 & 0\\
0 & 0 & 0 & \mu
\end{bmatrix} \quad \alpha, \mu > 0 \quad \beta \geq 0 \quad \gamma \in \mathbb{R}.$$
We distinguish two cases associated to this metric (the other cases are similar to these two cases)
\begin{enumerate}
	\item If $\beta = \gamma = 0$, then the left invariant Riemannian metric associated to $M$ is 
	$$\langle ., .\rangle_1 = \left( M^{-1}\right)^T\left( M^{-1}\right) = \begin{bmatrix}
	1 & 0 & 0 & 0\\
	0 & \frac{1}{\alpha^2} & 0 & 0\\
	0 & 0 & 1 & 0\\
	0 & 0 & 0 & \frac{1}{\mu^2}
	\end{bmatrix}$$
	\item If $\beta \neq 0$ and $\gamma = 0$, then our left invariant Riemannian metric is 
	$$\langle ., .\rangle_2 = \left( M^{-1}\right)^T\left( M^{-1}\right) = \begin{bmatrix}
	1 & 0 & -\beta & 0\\
	0 & \frac{1}{\alpha^2} & 0 & 0\\
	-\beta & 0 & 1 + \beta^2 & 0\\
	0 & 0 & 0 & \frac{1}{\mu^2}
	\end{bmatrix}.$$
\end{enumerate}
\subsection{Left invariant Riemannian metrics on $\operatorname{Sol}_1^4$}
By theorem 5.1 in \cite{van2017metrics}, any metric on $\operatorname{Sol}_1^4$ is equivalent, up to isometry, to the following metric
$$M = \begin{bmatrix}
\alpha & \beta & \gamma & 0\\
0 & 1 & \mu & 0\\
0 & 0 & 1 & 0\\
0 & 0 & 0 & \nu
\end{bmatrix} \quad \alpha, \nu > 0 \quad \beta, \mu \geq 0 \quad \gamma \in \mathbb{R}.$$
We distinguish three cases associated to this metric (the other cases are similar to these three cases)
\begin{enumerate}
	\item If $\beta = \gamma = \mu = 0$, then the left invariant Riemannian metric associated to $M$ is 
	$$\langle ., .\rangle_1 = \left( M^{-1}\right)^T\left( M^{-1}\right) = \begin{bmatrix}
	\frac{1}{\alpha^2} & 0 & 0 & 0\\
	0 & 1 & 0 & 0\\
	0 & 0 & 1 & 0\\
	0 & 0 & 0 & \frac{1}{\nu^2}
	\end{bmatrix}$$
	\item If $\beta > 0$ and $\gamma = \mu = 0$, then our left invariant Riemannian metric is
	$$\langle ., .\rangle_2 = \left( M^{-1}\right)^T\left( M^{-1}\right) = \begin{bmatrix}
	\frac{1}{\alpha^2} & \frac{-\beta}{\alpha^2} & 0 & 0\\
	\\
	\frac{-\beta}{\alpha^2} & 1 + (\frac{\beta}{\alpha})^2 & 0 & 0\\
	\\
	0 & 0 & 1 & 0\\
	\\
	0 & 0 & 0 & \frac{1}{\nu^2}
	\end{bmatrix}$$
	\item If $\beta, \mu > 0$ and $\gamma = 0$, then our left invariant Riemannian metric is
	$$\langle ., .\rangle_3 = \left( M^{-1}\right)^T\left( M^{-1}\right) =  \begin{bmatrix}
	\frac{1}{\alpha^2} & \frac{-\beta}{\alpha^2} & \frac{\beta\mu}{\alpha^2} & 0\\
	\\
	\frac{-\beta}{\alpha^2} & 1 + (\frac{\beta}{\alpha})^2 & \frac{-\beta^2\mu}{\alpha^2} - \mu & 0\\
	\\
	\frac{\beta\mu}{\alpha^2} & \frac{-\beta^2\mu}{\alpha^2} - \mu & (\frac{\beta\mu}{\alpha})^2 + \mu^2 + 1 & 0\\
	\\
	0 & 0 & 0 & \frac{1}{\nu^2}
	\end{bmatrix}.$$
\end{enumerate}
\section{The isometry groups}
Let us recall some well known and important results about the isometry groups left invariant Riemannian metrics on Lie groups.\\
Assume that our Lie group $G$ is compact, then for any left invariant Riemannian metric on $G$. The isometry group $\operatorname{Isom}(G, g)$ is also compact \cite{myers1939group}.\\
If $G$ is compact simple and if $g$ is a left invariant Riemannian metric on $G$. Then $\operatorname{Isom}_0(G, g) \subset L(G)R(G)$ \cite{ochiai1976group}. This means that any isometry in the connected component of the identity $\operatorname{Isom}_0(G, g)$ decomposes into a left translation and a right translation.\\
Let $\theta$ be an automorphism of $G$ and let $g$ be a left invariant Riemannian metric on $G$. The pullback $\theta^{\ast}g$ of $g$ by $\theta$ is defined by:
$\theta^{\ast}g(u, v) = g\left(\theta_{\ast}^{-1}(u), \theta_{\ast}^{-1}(v)\right) \; \forall u, v \in \mathfrak{g}$ where $\theta_{\ast}$ is the differential of $\theta$ at the identity element $e$ of $G$. This equation is equivalent to the following matrix equation $[g] = [\theta_{\ast}]^t[\theta^{\ast}g][\theta_{\ast}]$. Hence \cite{ha2012isometry}
$$\operatorname{Aut}(G)_{g} = \left\lbrace \theta \in \operatorname{Aut}(G)\,\, /\,\, [g] = [\theta_{\ast}]^t[g][\theta_{\ast}] \right\rbrace.$$
Let us show that when two metrics are equivalent up to a diffeomorphism, their groups of isometries are isomorphic. We will use the general definition of the pullback of a metric by a diffeomorphism, which is given by: $$\left( \theta^{\ast}g\right)_p(u, v) = g_{\theta(p)}\left(d_p\theta(u), d_p\theta(v)\right) \quad \forall p \in G,\quad \forall u, v \in T_pG.$$
This recovers the particular case of the pullback by an automorphism defined above. In fact, for the group of isometries, it is sufficient to describe the isotropy subgroup at the identity element $e$ of $G$, and our strategy makes sense when two metrics are equivalent up to an automorphism. One can easily verify that: $\left( \theta \circ \phi\right)^{\ast}g = \phi^{\ast}\theta^{\ast}g$.
\begin{theorem}
	Let $(G, g)$ be a Riemannian Lie group and let $\theta$ be a diffeomorphism of $G$, then the group of isometries of $G$ with respect to the Riemannian metric $g$ denoted $I(G, g)$ is isomorphic to the group of isometries of $G$ with respect to the Riemannian metric $\theta^{\ast}g$ denoted $I(G, \theta^{\ast}g)$. \label{th.5}
\end{theorem}
\begin{proof}
	Define the map 
	$$ \begin{array}{rcl}
	\psi_{\theta} : I(G, g)&\longrightarrow& I(G, \theta^{\ast}g)\\
	f  &\longmapsto& \theta^{-1}\circ f \circ \theta
	\end{array} $$
	$\bullet$ $\psi_{\theta}$ is well defined, in fact we have 
	$$\left(\theta^{-1}\circ f \circ \theta\right)^{\ast}\theta^{\ast}g = \theta^{\ast}f^{\ast}(\theta^{-1})^{\ast}\theta^{\ast}g = \theta^{\ast}f^{\ast}g = \theta^{\ast}g $$
	Hence $\psi_{\theta}$ is well defined.\\
	$\bullet$ $\psi_{\theta}$ is a group homomorphism: for $f, g \in I(G, g)$ we have
	\begin{eqnarray*}
		\psi(f \circ g) &=& \theta^{-1} \circ f \circ g \circ \theta\\
		&=& \theta^{-1} \circ f \circ \theta \circ \theta^{-1} \circ g \circ \theta\\
		&=& \psi(f) \circ \psi(g).
	\end{eqnarray*}
	$\bullet$ The map $\psi_{\theta}$ is clearly injective.\\
	$\bullet$ We now aim to show that $\psi_{\theta}$ is surjective, let $u \in I(G, \theta^{\ast}g)$, put $v = \theta \circ u \circ \theta^{-1}$ then 
	$$\psi_{\theta}(v) = \theta^{-1} \circ \theta \circ u \circ \theta^{-1} \circ \theta = u $$
	in addition $v^{\ast}g = \big(\theta \circ u \circ \theta^{-1}\big)^{\ast}g = (\theta^{-1})^{\ast}u^{\ast}\theta^{\ast}g = (\theta^{-1})^{\ast}\theta^{\ast}g = g.$\\
	Hence $v \in I(G, g)$ and $\psi_{\theta}(v) = u$. Consequently $\psi_{\theta}$ is a group isomorphism.
\end{proof}
According to this theorem, one only needs to describe the group of isometries of the representative metrics of non-equivalent classes of metrics on Lie groups.
\subsection{Nilpotent case}
Using the result $(\ref{formula1})$, we give the isometry group of our two nilpotent four-dimensional Lie groups.
\subsubsection{Isometry group of $\operatorname{Nil}^3 \times \mathbb{R}$}
\begin{theorem}
	The isometry group of $\operatorname{Nil}^3 \times \mathbb{R}$ is given by
	$$\operatorname{Isom}\left(\operatorname{Nil}^3 \times \mathbb{R}, \langle ., .\rangle\right)  \cong \left(\operatorname{Nil}^3 \times \mathbb{R}\right) \rtimes \left(\operatorname{diag}\left\lbrace \pm1, \pm1, \operatorname{O}(2) \right\rbrace \right)$$
\end{theorem}
\begin{proof}
	The automorphism group $\operatorname{Aut}\left(\operatorname{Nil}^3 \times \mathbb{R}\right)$ consists of elements of the form \cite{van2017metrics}
	$$\begin{bmatrix}
	ad - bc & x & y & u\\
	0 & e & z & v\\
	0 & 0 & a & b\\
	0 & 0 & c & d
	\end{bmatrix} \quad \begin{bmatrix}
	a & b\\
	c & d
	\end{bmatrix} \in \operatorname{GL}(2, \mathbb{R}) \quad e \in \mathbb{R}^{\ast} \quad x, y, z, u, v \in \mathbb{R}.$$
	Let $A = \begin{bmatrix}
	ad - bc & x & y & u\\
	0 & e & z & v\\
	0 & 0 & a & b\\
	0 & 0 & c & d
	\end{bmatrix} \in \operatorname{Aut}\left(\operatorname{Nil}^3 \times \mathbb{R}\right)$, then
	\begin{eqnarray*}
		A \in \operatorname{Aut}\left(\operatorname{Nil}^3 \times \mathbb{R}\right)_{\langle ., .\rangle} &\Leftrightarrow& A^t\langle ., .\rangle A = \langle ., .\rangle\\
		&\Leftrightarrow& \begin{bmatrix}
			a & b\\
			c & d
		\end{bmatrix} \in \operatorname{O}(2), \; e = \pm1, \; x = y = z = u = v = 0.
	\end{eqnarray*}
	Hence $$\operatorname{Isom}\left(\operatorname{Nil}^3 \times \mathbb{R}, \langle ., .\rangle\right)  \cong \left(\operatorname{Nil}^3 \times \mathbb{R}\right) \rtimes \left(\operatorname{diag}\left\lbrace \pm1, \pm1, \operatorname{O}(2)\right\rbrace \right).$$
\end{proof}
\subsubsection{Isometry group of $\operatorname{Nil}^4$}
\begin{theorem}
	The isometry group of $\operatorname{Nil}^4$ is given by
	\begin{eqnarray*}
		\operatorname{Isom}\left(\operatorname{Nil}^4, \langle ., .\rangle_1\right)  &\cong& \operatorname{Nil}^4 \rtimes \left(\mathbb{Z}_2\right)^2\\
		\operatorname{Isom}\left(\operatorname{Nil}^4, \langle ., .\rangle_2\right)  &\cong& \operatorname{Nil}^4 \rtimes (\mathbb{Z}_2).
	\end{eqnarray*}
\end{theorem}
\begin{proof}
	The automorphism group $\operatorname{Aut}\left(\operatorname{Nil}^4\right)$ consists of elements of the form \cite{van2017metrics}
	$$\begin{bmatrix}
	ad^2 & ed & x & y\\
	0 & ad & e & z\\
	0 & 0 & a & b\\
	0 & 0 & 0 & d
	\end{bmatrix} \quad a, d \in \mathbb{R}^{\ast} \quad b, e, x, y, z \in \mathbb{R}.$$
	Let $A = \begin{bmatrix}
	ad^2 & ed & x & y\\
	0 & ad & e & z\\
	0 & 0 & a & b\\
	0 & 0 & 0 & d
	\end{bmatrix} \in \operatorname{Aut}\left(\operatorname{Nil}^4\right)$, then
	\begin{eqnarray*}
		& A \in \operatorname{Aut}\left(\operatorname{Nil}^4\right)_{\langle ., .\rangle_1}\hspace{10cm}\\ \Leftrightarrow & A^t\langle ., .\rangle_1A = \langle ., .\rangle_1\hspace{10cm}\\
		\Leftrightarrow & A \in \left\lbrace I_4, \begin{bmatrix}
			1 & 0 & 0 & 0\\
			0 & -1 & 0 & 0\\
			0 & 0 & 1 & 0\\
			0 & 0 & 0 & -1
		\end{bmatrix}, \begin{bmatrix}
		-1 & 0 & 0 & 0\\
		0 & -1 & 0 & 0\\
		0 & 0 & -1 & 0\\
		0 & 0 & 0 & 1
	\end{bmatrix}, \begin{bmatrix}
	-1 & 0 & 0 & 0\\
	0 & 1 & 0 & 0\\
	0 & 0 & -1 & 0\\
	0 & 0 & 0 & -1
\end{bmatrix} \right\rbrace \cong \left(\mathbb{Z}_2\right)^2.
\end{eqnarray*}
Hence $$\operatorname{Isom}\left(\operatorname{Nil}^4, \langle ., .\rangle_1\right)  \cong \operatorname{Nil}^4 \rtimes \left(\mathbb{Z}_2\right)^2.$$
For the metric $\langle ., .\rangle_2$, we have $A \in \operatorname{Aut}\left(\operatorname{Nil}^4\right)_{\langle ., .\rangle_2} \Leftrightarrow A^t\langle ., .\rangle_2A = \langle ., .\rangle_2$. If we see the third component of the first and the second lines of the matrices $A^t\langle ., .\rangle_2A$ and $\langle ., .\rangle_2$, we obtain the following system
$$\left\lbrace\begin{array}{lll}
\frac{ad^2}{\alpha^2} x - \frac{ad^2\beta}{\alpha^2\gamma} e = 0\\
\\
\left[\frac{ed}{\alpha^2} - \frac{ad\beta}{\alpha^2\gamma}\right]x + \left[\frac{-ed\beta}{\alpha^2\gamma} + ad\left( \frac{1}{\gamma^2} + \frac{\beta^2}{\alpha^2\gamma^2}\right) \right] e = 0
\end{array}\right.$$
If we see the last component of the first and the second lines of the matrices $A^t\langle ., .\rangle_2A$ and $\langle ., .\rangle_2$, we obtain the same system but in $y$ and $z$. The determinant of the associated matrix to this system is equal to $\frac{a^2d^3}{\alpha^2\gamma^2} \neq 0$, hence $x = e = y= z = 0$ in $A$. Now since the form of $A$ is restricted, one simply gets that
\begin{eqnarray*}
	A \in \operatorname{Aut}\left(\operatorname{Nil}^4\right)_{\langle ., .\rangle_2} &\Leftrightarrow& A^t\langle ., .\rangle_2A = \langle ., .\rangle_2\\
	&\Leftrightarrow& A \in \left\lbrace \begin{bmatrix}
		1 & 0 & 0 & 0\\
		0 & 1 & 0 & 0\\
		0 & 0 & 1 & 0\\
		0 & 0 & 0 & 1
	\end{bmatrix}, \begin{bmatrix}
	-1 & 0 & 0 & 0\\
	0 & -1 & 0 & 0\\
	0 & 0 & -1 & 0\\
	0 & 0 & 0 & 1
\end{bmatrix}\right\rbrace \cong \mathbb{Z}_2.
\end{eqnarray*}
Hence $\operatorname{Aut}\left(\operatorname{Nil}^4\right)_{\langle ., .\rangle_2} \cong \mathbb{Z}_2$. Therefore $\operatorname{Isom}\left(\operatorname{Nil}^4, \langle ., .\rangle_2\right)  \cong \operatorname{Nil}^4 \rtimes \mathbb{Z}_2.$
\end{proof}
\begin{remark}
	Note that when we find the group of isometric automorphisms of the first class of metrics, the problem becomes easy, because the group of isometric automorphisms of the first class of metrics is the maximal one. This fact has a significant impact on the rest of the article.
\end{remark}
\subsection{Solvable $(R)$-type case}
Using the result $(\ref{formula2})$, we give the isomerty groups of our five solvable Lie groups of type $(R)$.
\subsubsection{Isometry group of $\operatorname{Sol_{m, n}^4}$}
The automorphism group of $\operatorname{Sol}_{m, n}^4$ consists of elements of the form \cite{van2017metrics}
$$\begin{bmatrix}
a & 0 & 0 & x\\
0 & b & 0 & y\\
0 & 0 & c & z\\
0 & 0 & 0 & 1
\end{bmatrix} \qquad a, b, c \in \mathbb{R}^{\ast}, \quad x, y, z \in \mathbb{R}. $$
\begin{theorem}
	The group of isometric automorphisms of $\operatorname{Sol}_{m, n}^4$ is given by
	\begin{eqnarray*}
		\operatorname{Aut}(\operatorname{Sol}_{m, n}^4)_{\langle ., .\rangle_1} &\cong& (\mathbb{Z}_2)^3\\
		\operatorname{Aut}(\operatorname{Sol}_{m, n}^4)_{\langle ., .\rangle_2} &\cong& (\mathbb{Z}_2)^2\\
		\operatorname{Aut}(\operatorname{Sol}_{m, n}^4)_{\langle ., .\rangle_3} &\cong& \mathbb{Z}_2. \label{Isom1}
	\end{eqnarray*}
\end{theorem}
\begin{proof}
	Let $A = \begin{bmatrix}
	a & 0 & 0 & x\\
	0 & b & 0 & y\\
	0 & 0 & c & z\\
	0 & 0 & 0 & 1 \end{bmatrix} \in \operatorname{Aut}(\operatorname{Sol_{m, n}^4})$. Then
	$$A \in \operatorname{Aut}(\operatorname{Sol_{m, n}^4})_{\langle ., .\rangle_1} \Leftrightarrow A^t\langle ., .\rangle_1A = \langle ., .\rangle_1 \Leftrightarrow A \in \left\lbrace \begin{bmatrix}
	\pm1 & 0 & 0 & 0\\
	0 & \pm1 & 0 & 0\\
	0 & 0 & \pm1 & 0\\
	0 & 0 & 0 & 1 \end{bmatrix} \right\rbrace \cong (\mathbb{Z}_2)^3.
	$$
	For the metric $\langle ., .\rangle_2$, one gets that
	\begin{eqnarray*}
		A \in \operatorname{Aut}(\operatorname{Sol_{m, n}^4})_{\langle ., .\rangle_2} &\Leftrightarrow& A^t\langle ., .\rangle_2A = \langle ., .\rangle_2\\
		&\Leftrightarrow& z = 0 \quad a = b = \pm1 \quad c = \pm1
	\end{eqnarray*}
	and $x, y$ verify the following system
	$$\left\lbrace\begin{array}{lll}
	x - \alpha y = 0\\
	-\alpha x + \left(1 + \alpha^2\right)y  = 0
	\end{array}\right.$$
	This implies that $x = y = 0$. Hence
	\begin{eqnarray*}
		\operatorname{Aut}(\operatorname{Sol_{m, n}^4})_{\langle ., .\rangle_2} &=& \left\lbrace I_4, \begin{bmatrix}
			1 & 0 & 0 & 0\\
			0 & 1 & 0 & 0\\
			0 & 0 & -1 & 0\\
			0 & 0 & 0 & 1
		\end{bmatrix}, \begin{bmatrix}
		-1 & 0 & 0 & 0\\
		0 & -1 & 0 & 0\\
		0 & 0 & 1 & 0\\
		0 & 0 & 0 & 1
	\end{bmatrix}, \begin{bmatrix}
	-1 & 0 & 0 & 0\\
	0 & -1 & 0 & 0\\
	0 & 0 & -1 & 0\\
	0 & 0 & 0 & 1
\end{bmatrix} \right\rbrace\\
&\cong& (\mathbb{Z}_2)^2.
\end{eqnarray*}
Finally, one gets that
\begin{eqnarray*}
	A \in \operatorname{Aut}(\operatorname{Sol_{m, n}^4})_{\langle ., .\rangle_3} &\Leftrightarrow& A^t\langle ., .\rangle_3A = \langle ., .\rangle_3\\
	&\Leftrightarrow& A \in \left\lbrace I_4, \begin{bmatrix}
		-1 & 0 & 0 & 0\\
		0 & -1 & 0 & 0\\
		0 & 0 & -1 & 0\\
		0 & 0 & 0 & 1
	\end{bmatrix} \right\rbrace \cong \mathbb{Z}_2.
\end{eqnarray*}
\end{proof}
\begin{corollary}
	The isometry group of $\operatorname{Sol}_{m, n}^4$ is given by
	\begin{eqnarray*}
		\operatorname{Isom}\left(\operatorname{Sol}_{m, n}^4, \langle ., .\rangle_1\right)  &\cong& \operatorname{Sol}_{m, n}^4 \rtimes (\mathbb{Z}_2)^3\\
		\operatorname{Isom}\left(\operatorname{Sol}_{m, n}^4, \langle ., .\rangle_2\right)  &\cong& \operatorname{Sol}_{m, n}^4 \rtimes (\mathbb{Z}_2)^2\\
		\operatorname{Isom}\left(\operatorname{Sol}_{m, n}^4, \langle ., .\rangle_3\right)  &\cong& \operatorname{Sol}_{m, n}^4 \rtimes \mathbb{Z}_2.
	\end{eqnarray*}
\end{corollary}
\subsubsection{Isometry group of $\operatorname{Sol}^3 \times \mathbb{R}$}
We denote by $\mathbf{D}(4)$ the dihedral group of order $8$.
\begin{theorem}
	The group of isometric automorphisms of $\operatorname{Sol}^3 \times \mathbb{R}$ is given by
	\begin{eqnarray*}
		\operatorname{Aut}\left(\operatorname{Sol}^3 \times \mathbb{R}\right)_{\langle ., .\rangle_1} &\cong& \mathbf{D}(4) \times \mathbb{Z}_{2}\\
		\operatorname{Aut}\left(\operatorname{Sol}^3 \times \mathbb{R}\right)_{\langle ., .\rangle_2} &\cong& (\mathbb{Z}_2)^2\\
		\operatorname{Aut}\left(\operatorname{Sol}^3 \times \mathbb{R}\right)_{\langle ., .\rangle_3} &\cong& \mathbb{Z}_2. \label{Isom2}
	\end{eqnarray*}
\end{theorem}
\begin{proof}
	The group of automorphisms of $\operatorname{Sol}^3 \times \mathbb{R}$ consists of all elements of $\operatorname{GL}(4, \mathbb{R})$ of the form \cite{van2017metrics}
	$$\begin{bmatrix}
	e & 0 & 0 & x\\
	0 & a & 0 & y\\
	0 & 0 & d & z\\
	0 & 0 & 0 & 1 \end{bmatrix} \quad \text{or}\quad \begin{bmatrix}
	e & 0 & 0 & x\\
	0 & 0 & b & y\\
	0 & c & 0 & z\\
	0 & 0 & 0 & -1 \end{bmatrix} \quad a, b, c, d, e \in \mathbb{R}^{\ast}, x, y, z \in \mathbb{R}.$$
	Hence, one gets that
	\begin{eqnarray*}
		& A \in \operatorname{Aut}\left(\operatorname{Sol}^3 \times \mathbb{R}\right)_{\langle ., .\rangle_1}\hspace{10cm}\\ \Leftrightarrow & A^t\langle ., .\rangle_1A = \langle ., .\rangle_1\hspace{10cm}\\
		\Leftrightarrow & A \in   (\mathbb{Z}_2)^3 \cup \left\lbrace \begin{bmatrix}
			1 & 0 & 0 & 0\\
			0 & 0 & 1 & 0\\
			0 & 1 & 0 & 0\\
			0 & 0 & 0 & -1
		\end{bmatrix}, \begin{bmatrix}
		1 & 0 & 0 & 0\\
		0 & 0 & 1 & 0\\
		0 & -1 & 0 & 0\\
		0 & 0 & 0 & -1
	\end{bmatrix}, \begin{bmatrix}
	1 & 0 & 0 & 0\\
	0 & 0 & -1 & 0\\
	0 & 1 & 0 & 0\\
	0 & 0 & 0 & -1
\end{bmatrix} \right\rbrace \hspace{3cm}\\
& \cup \left\lbrace \begin{bmatrix}
	1 & 0 & 0 & 0\\
	0 & 0 & -1 & 0\\
	0 & -1 & 0 & 0\\
	0 & 0 & 0 & -1
\end{bmatrix}, \begin{bmatrix}
-1 & 0 & 0 & 0\\
0 & 0 & 1 & 0\\
0 & 1 & 0 & 0\\
0 & 0 & 0 & -1
\end{bmatrix}, \begin{bmatrix}
-1 & 0 & 0 & 0\\
0 & 0 & 1 & 0\\
0 & -1 & 0 & 0\\
0 & 0 & 0 & -1
\end{bmatrix} \right\rbrace \hspace{0.5cm} \\
&  \cup \left\lbrace \begin{bmatrix}
	-1 & 0 & 0 & 0\\
	0 & 0 & -1 & 0\\
	0 & 1 & 0 & 0\\
	0 & 0 & 0 & -1
\end{bmatrix}, \begin{bmatrix}
-1 & 0 & 0 & 0\\
0 & 0 & -1 & 0\\
0 & -1 & 0 & 0\\
0 & 0 & 0 & -1
\end{bmatrix} \right\rbrace . \hspace{2.5cm}
\end{eqnarray*}
Therefore $\operatorname{Aut}\left(\operatorname{Sol}^3 \times \mathbb{R}\right)_{\langle ., .\rangle_1}$ is a group of order 16 such that :\\
The following $11$ elements are of order $2$
$$\hspace{1cm} \begin{bmatrix}
1 & 0 & 0 & 0\\
0 & 1 & 0 & 0\\
0 & 0 & -1 & 0\\
0 & 0 & 0 & 1
\end{bmatrix} \begin{bmatrix}
1 & 0 & 0 & 0\\
0 & -1 & 0 & 0\\
0 & 0 & 1 & 0\\
0 & 0 & 0 & 1
\end{bmatrix} \begin{bmatrix}
1 & 0 & 0 & 0\\
0 & -1 & 0 & 0\\
0 & 0 & -1 & 0\\
0 & 0 & 0 & 1
\end{bmatrix} \begin{bmatrix}
-1 & 0 & 0 & 0\\
0 & 1 & 0 & 0\\
0 & 0 & -1 & 0\\
0 & 0 & 0 & 1
\end{bmatrix} $$
$$\hspace{1.3cm}\begin{bmatrix}
-1 & 0 & 0 & 0\\
0 & 1 & 0 & 0\\
0 & 0 & 1 & 0\\
0 & 0 & 0 & 1
\end{bmatrix} \begin{bmatrix}
-1 & 0 & 0 & 0\\
0 & -1 & 0 & 0\\
0 & 0 & 1 & 0\\
0 & 0 & 0 & 1
\end{bmatrix} \begin{bmatrix}
-1 & 0 & 0 & 0\\
0 & -1 & 0 & 0\\
0 & 0 & -1 & 0\\
0 & 0 & 0 & 1
\end{bmatrix} \begin{bmatrix}
1 & 0 & 0 & 0\\
0 & 0 & 1 & 0\\
0 & 1 & 0 & 0\\
0 & 0 & 0 & -1
\end{bmatrix} $$
$$\begin{bmatrix}
1 & 0 & 0 & 0\\
0 & 0 & -1 & 0\\
0 & -1 & 0 & 0\\
0 & 0 & 0 & -1
\end{bmatrix} \begin{bmatrix}
-1 & 0 & 0 & 0\\
0 & 0 & 1 & 0\\
0 & 1 & 0 & 0\\
0 & 0 & 0 & -1
\end{bmatrix} \begin{bmatrix}
-1 & 0 & 0 & 0\\
0 & 0 & -1 & 0\\
0 & -1 & 0 & 0\\
0 & 0 & 0 & -1
\end{bmatrix}.$$
The following $4$ elements are of order $4$
$$\begin{bmatrix}
1 & 0 & 0 & 0\\
0 & 0 & 1 & 0\\
0 & -1 & 0 & 0\\
0 & 0 & 0 & -1
\end{bmatrix} 
\begin{bmatrix}
1 & 0 & 0 & 0\\
0 & 0 & -1 & 0\\
0 & 1 & 0 & 0\\
0 & 0 & 0 & -1
\end{bmatrix} \begin{bmatrix}
-1 & 0 & 0 & 0\\
0 & 0 & 1 & 0\\
0 & -1 & 0 & 0\\
0 & 0 & 0 & -1
\end{bmatrix} \begin{bmatrix}
-1 & 0 & 0 & 0\\
0 & 0 & -1 & 0\\
0 & 1 & 0 & 0\\
0 & 0 & 0 & -1
\end{bmatrix}.$$
Therefore $$\operatorname{Aut}\left(\operatorname{Sol}^3 \times \mathbb{R}\right)_{\langle ., .\rangle_1} \cong \mathbf{D}(4) \times \mathbb{Z}_2.$$
For the metric $\langle ., .\rangle_2$, all the automorphisms of the form $\begin{bmatrix}
e & 0 & 0 & x\\
0 & 0 & b & y\\
0 & c & 0 & z\\
0 & 0 & 0 & -1 \end{bmatrix}$ do not preserves our metric. Then
$$\operatorname{Aut}\left(\operatorname{Sol}^3 \times \mathbb{R}\right)_{\langle ., .\rangle_2} = \operatorname{Aut}(\operatorname{Sol}_{m, n}^4)_{\langle ., .\rangle_2} \cong (\mathbb{Z}_2)^2.$$
The same for the third metric, we have
$$\operatorname{Aut}\left(\operatorname{Sol}^3 \times \mathbb{R}\right)_{\langle ., .\rangle_3} = \operatorname{Aut}(\operatorname{Sol}_{m, n}^4)_{\langle ., .\rangle_3} \cong \mathbb{Z}_2. $$
\end{proof}
\begin{corollary}
	The isometry group of $\operatorname{Sol}^3 \times \mathbb{R}$ is given by
	\begin{eqnarray*}
		\operatorname{Isom}\left(\operatorname{Sol}^3 \times \mathbb{R}, \langle ., .\rangle_1\right)  &\cong& \left(\operatorname{Sol}^3 \times \mathbb{R}\right) \rtimes \left(\mathbf{D}(4) \times \mathbb{Z}_{2}\right)\\
		\operatorname{Isom}\left(\operatorname{Sol}^3 \times \mathbb{R}, \langle ., .\rangle_2\right)  &\cong& \left(\operatorname{Sol}^3 \times \mathbb{R}\right) \rtimes (\mathbb{Z}_{2})^2\\
		\operatorname{Isom}\left(\operatorname{Sol}^3 \times \mathbb{R}, \langle ., .\rangle_3\right)  &\cong& \left(\operatorname{Sol}^3 \times \mathbb{R}\right) \rtimes \mathbb{Z}_{2}.
	\end{eqnarray*}
\end{corollary}
\subsubsection{Isometry group of $\operatorname{Sol}_0^4$}
\begin{theorem}
	The isometry group of $\operatorname{Sol}_0^4$ is given by
	\begin{eqnarray*}
		\operatorname{Isom}\left(\operatorname{Sol}_0^4, \langle ., .\rangle_1\right)  &\cong& \operatorname{Sol}_0^4 \rtimes \left(\operatorname{diag}\left\lbrace \operatorname{O}(2), \pm1, 1\right\rbrace \right)\\
		\operatorname{Isom}\left(\operatorname{Sol}_0^4, \langle ., .\rangle_2\right)  &\cong& \operatorname{Sol}_0^4 \rtimes (\mathbb{Z}_2)^2.
		\label{Isom3}
	\end{eqnarray*}
\end{theorem}
\begin{proof}
	The automorphism group $\operatorname{Aut}(\operatorname{Sol_0^4})$ consists of elements of the form \cite{van2017metrics}
	$$\begin{bmatrix}
	a & b & 0 & x\\
	c & d & 0 & y\\
	0 & 0 & e & z\\
	0 & 0 & 0 & 1
	\end{bmatrix} \quad \begin{bmatrix}
	a & b\\
	c & d
	\end{bmatrix} \in \operatorname{GL}(2, \mathbb{R}) \quad e \in \mathbb{R}^{\ast} \quad x, y, z \in \mathbb{R}.$$
	Let $A = \begin{bmatrix}
	a & b & 0 & x\\
	c & d & 0 & y\\
	0 & 0 & e & z\\
	0 & 0 & 0 & 1
	\end{bmatrix} \in \operatorname{Aut}(\operatorname{Sol_0^4})$, then
	\begin{eqnarray*}
		A \in \operatorname{Aut}(\operatorname{Sol_0^4})_{\langle ., .\rangle_1} &\Leftrightarrow& A^t\langle ., .\rangle_1A = \langle ., .\rangle_1\\
		&\Leftrightarrow& \begin{bmatrix}
			a & b\\
			c & d
		\end{bmatrix} \in \operatorname{O}(2), e = \pm1, x = y = z = 0.
	\end{eqnarray*}
	Hence $$\operatorname{Isom}\left(\operatorname{Sol}_0^4, \langle ., .\rangle_1\right)  \cong \operatorname{Sol}_0^4 \rtimes \left(\operatorname{diag}\left\lbrace \operatorname{O}(2), \pm1, 1\right\rbrace \right).$$
	For the metric $\langle ., .\rangle_2$, one gets that
	\begin{eqnarray*}
		A \in \operatorname{Aut}(\operatorname{Sol_0^4})_{\langle ., .\rangle_2} &\Leftrightarrow& A^t\langle ., .\rangle_2A = \langle ., .\rangle_2\\
		&\Leftrightarrow& b = c = y = 0 \quad a = e = \pm1 \quad d = \pm1
	\end{eqnarray*}
	and $x, z$ verify the following system
	$$\left\lbrace\begin{array}{lll}
	x - \alpha z = 0\\
	-\alpha x + \left(1 + \alpha^2\right)z  = 0
	\end{array}\right.$$
	This implies that $x = z = 0$. Hence
	\begin{eqnarray*}
		\operatorname{Aut}(\operatorname{Sol_0^4})_{\langle ., .\rangle_2} &=& \left\lbrace I_4, \begin{bmatrix}
			1 & 0 & 0 & 0\\
			0 & -1 & 0 & 0\\
			0 & 0 & 1 & 0\\
			0 & 0 & 0 & 1
		\end{bmatrix}, \begin{bmatrix}
		-1 & 0 & 0 & 0\\
		0 & 1 & 0 & 0\\
		0 & 0 & -1 & 0\\
		0 & 0 & 0 & 1
	\end{bmatrix}, \begin{bmatrix}
	-1 & 0 & 0 & 0\\
	0 & -1 & 0 & 0\\
	0 & 0 & -1 & 0\\
	0 & 0 & 0 & 1
\end{bmatrix} \right\rbrace\\
&\cong& (\mathbb{Z}_2)^2.
\end{eqnarray*}
Consequently $$\operatorname{Isom}\left(\operatorname{Sol}_0^4, \langle ., .\rangle_2\right)  \cong \operatorname{Sol}_0^4 \rtimes (\mathbb{Z}_2)^2. $$
\end{proof}
\subsubsection{Isometry group of $\operatorname{Sol}_0^{' 4}$}
\begin{theorem}
	The isometry group of $\operatorname{Sol}_0^{' 4}$ is given by
	\begin{eqnarray*}
		\operatorname{Isom}\left(\operatorname{Sol}_0^{' 4}, \langle ., .\rangle_1\right)  &\cong& \operatorname{Sol}_0^{' 4} \rtimes (\mathbb{Z}_2)^2\\
		\operatorname{Isom}\left(\operatorname{Sol}_0^{' 4}, \langle ., .\rangle_2\right)  &\cong& \operatorname{Sol}_0^{' 4} \rtimes \mathbb{Z}_2. \label{Isom4}
	\end{eqnarray*}
\end{theorem}
\begin{proof}
	Let $A = \begin{bmatrix}
	a & b & 0 & x\\
	0 & a & 0 & y\\
	0 & 0 & e & z\\
	0 & 0 & 0 & 1 \end{bmatrix} \in \operatorname{Aut}\left(\operatorname{Sol}_0^{' 4}\right)$ (see \cite{van2017metrics}). Then
	\begin{eqnarray*}
		& A \in \operatorname{Aut}\left(\operatorname{Sol}_0^{' 4}\right)_{\langle ., .\rangle_1}\hspace{10cm}\\ \Leftrightarrow & A^t\langle ., .\rangle_1A = \langle ., .\rangle_1\hspace{10cm}\\
		\Leftrightarrow & A \in \left\lbrace I_4, \begin{bmatrix}
			1 & 0 & 0 & 0\\
			0 & 1 & 0 & 0\\
			0 & 0 & -1 & 0\\
			0 & 0 & 0 & 1
		\end{bmatrix}, \begin{bmatrix}
		-1 & 0 & 0 & 0\\
		0 & -1 & 0 & 0\\
		0 & 0 & 1 & 0\\
		0 & 0 & 0 & 1
	\end{bmatrix}, \begin{bmatrix}
	-1 & 0 & 0 & 0\\
	0 & -1 & 0 & 0\\
	0 & 0 & -1 & 0\\
	0 & 0 & 0 & 1
\end{bmatrix} \right\rbrace.
\end{eqnarray*}
This is exactly the Klein 4-group, thus $\operatorname{Aut}\left(\operatorname{Sol}_0^{' 4}\right)_{\langle ., .\rangle_1} \cong (\mathbb{Z}_2)^2.$ Hence 
$$\operatorname{Isom}\left(\operatorname{Sol}_0^{' 4}, \langle ., .\rangle_1\right)  \cong \operatorname{Sol}_0^{' 4} \rtimes (\mathbb{Z}_2)^2.$$
For the metric $\langle ., .\rangle_2$, we have
\begin{eqnarray*}
	A \in \operatorname{Aut}\left(\operatorname{Sol}_0^{' 4}\right)_{\langle ., .\rangle_2} &\Leftrightarrow& A^t\langle ., .\rangle_2A = \langle ., .\rangle_2\\
	&\Leftrightarrow& A \in \left\lbrace I_4, \begin{bmatrix}
		-1 & 0 & 0 & 0\\
		0 & -1 & 0 & 0\\
		0 & 0 & -1 & 0\\
		0 & 0 & 0 & 1
	\end{bmatrix} \right\rbrace \cong \mathbb{Z}_2.
\end{eqnarray*}
Therefore 
$$\operatorname{Isom}\left(\operatorname{Sol}_0^{' 4}, \langle ., .\rangle_2\right)  \cong \operatorname{Sol}_0^{' 4} \rtimes \mathbb{Z}_2.$$
\end{proof}
\subsubsection{Isometry group of $\operatorname{Sol}_1^4$}
\begin{theorem}
	The group of isometric automorphisms of $\operatorname{Sol}_1^4$ is given by
	\begin{eqnarray*}
		\operatorname{Aut}\left(\operatorname{Sol}_1^4\right)_{\langle ., .\rangle_1} &\cong& \mathbf{D}(4)\\
		\operatorname{Aut}\left(\operatorname{Sol}_1^4\right)_{\langle ., .\rangle_2} &\cong& \mathbb{Z}_2\\
		\operatorname{Aut}\left(\operatorname{Sol}_1^4\right)_{\langle ., .\rangle_3} &=& \{I_4\}. \label{Isom5}
	\end{eqnarray*}
\end{theorem}
\begin{proof}
	The group of automorphisms of $\operatorname{Sol}_1^4$ consists of all elements of $\operatorname{GL}(4, \mathbb{R})$ of the form \cite{van2017metrics}
	$$\begin{bmatrix}
	ad & -aq & -dp & x\\
	0 & a & 0 & p\\
	0 & 0 & d & q\\
	0 & 0 & 0 & 1 \end{bmatrix} \quad \text{or}\quad \begin{bmatrix}
	-bc & cp & bq & x\\
	0 & 0 & b & p\\
	0 & c & 0 & q\\
	0 & 0 & 0 & -1 \end{bmatrix} \quad a, b, c, d \in \mathbb{R}^{\ast}, x, p, q \in \mathbb{R}.$$
	Hence, one gets that
	\begin{eqnarray*}
		& A \in \operatorname{Aut}\left(\operatorname{Sol}_1^4\right)_{\langle ., .\rangle_1}\hspace{13cm}\\ \Leftrightarrow & A^t\langle ., .\rangle_1A = \langle ., .\rangle_1\hspace{13cm}\\
		\Leftrightarrow & A \in \left\lbrace \begin{bmatrix}
			1 & 0 & 0 & 0\\
			0 & 1 & 0 & 0\\
			0 & 0 & 1 & 0\\
			0 & 0 & 0 & 1
		\end{bmatrix}, \begin{bmatrix}
		-1 & 0 & 0 & 0\\
		0 & 1 & 0 & 0\\
		0 & 0 & -1 & 0\\
		0 & 0 & 0 & 1
	\end{bmatrix}, \begin{bmatrix}
	-1 & 0 & 0 & 0\\
	0 & -1 & 0 & 0\\
	0 & 0 & 1 & 0\\
	0 & 0 & 0 & 1
\end{bmatrix}, \begin{bmatrix}
1 & 0 & 0 & 0\\
0 & -1 & 0 & 0\\
0 & 0 & -1 & 0\\
0 & 0 & 0 & 1
\end{bmatrix} \right\rbrace \hspace{2.4cm}\\
& \cup \left\lbrace \begin{bmatrix}
	-1 & 0 & 0 & 0\\
	0 & 0 & 1 & 0\\
	0 & 1 & 0 & 0\\
	0 & 0 & 0 & -1
\end{bmatrix}, \begin{bmatrix}
1 & 0 & 0 & 0\\
0 & 0 & -1 & 0\\
0 & 1 & 0 & 0\\
0 & 0 & 0 & -1
\end{bmatrix}, \begin{bmatrix}
1 & 0 & 0 & 0\\
0 & 0 & 1 & 0\\
0 & -1 & 0 & 0\\
0 & 0 & 0 & -1
\end{bmatrix}, \begin{bmatrix}
-1 & 0 & 0 & 0\\
0 & 0 & -1 & 0\\
0 & -1 & 0 & 0\\
0 & 0 & 0 & -1
\end{bmatrix} \right\rbrace \hspace{2cm}\\
\Leftrightarrow & A \in \left\langle \begin{bmatrix}
	-1 & 0 & 0 & 0\\
	0 & 1 & 0 & 0\\
	0 & 0 & -1 & 0\\
	0 & 0 & 0 & 1
\end{bmatrix}, \begin{bmatrix}
1 & 0 & 0 & 0\\
0 & 0 & -1 & 0\\
0 & 1 & 0 & 0\\
0 & 0 & 0 & -1
\end{bmatrix} \right\rangle  \cong \mathbf{D}(4). \hspace{7cm}
\end{eqnarray*}
Therefore $\operatorname{Aut}\left(\operatorname{Sol}_1^4\right)_{\langle ., .\rangle_1} \cong \mathbf{D}(4)$.\\
For the metric $\langle ., .\rangle_2$, all the automorphisms of the form $\begin{bmatrix}
-bc & cp & bq & x\\
0 & 0 & b & p\\
0 & c & 0 & q\\
0 & 0 & 0 & -1 \end{bmatrix}$ do not preserves our metric. In fact we have
$$\operatorname{Aut}\left(\operatorname{Sol}_1^4\right)_{\langle ., .\rangle_2} = \left\lbrace I_4, \operatorname{diag}\{-1, -1, 1, 1\}\right\rbrace \cong \mathbb{Z}_2.$$
The same for the third metric, we obtain $\operatorname{Aut}\left(\operatorname{Sol}_1^4\right)_{\langle ., .\rangle_3} = \{I_4\}. $
\end{proof}
\begin{corollary} 
	The isometry group of $\operatorname{Sol}_1^4$ is given by
	\begin{eqnarray*}
		\operatorname{Isom}\left(\operatorname{Sol}_1^4, \langle ., .\rangle_1\right)  &\cong& \operatorname{Sol}_1^4 \rtimes \mathbf{D}(4)\\
		\operatorname{Isom}\left(\operatorname{Sol}_1^4, \langle ., .\rangle_2\right)  &\cong& \operatorname{Sol}_1^4 \rtimes \mathbb{Z}_{2}\\
		\operatorname{Isom}\left(\operatorname{Sol}_1^4, \langle ., .\rangle_3\right)  &\cong& \operatorname{Sol}_1^4.
	\end{eqnarray*}
\end{corollary}


\end{document}